\documentclass[11pt]{amsart}
\usepackage{amssymb}
\usepackage{palatino}
\usepackage{tikz-cd}
\input amssym.def

\usepackage{xcolor}
\usepackage{amsmath, amsfonts}
\usepackage{amssymb}
\usepackage{amscd}
\usepackage[mathscr]{eucal}
\usepackage{palatino}
\usepackage{theoremref}
\usepackage{tikz}
\usepackage{float}
\usepackage{comment}

\newfont{\cyrr}{wncyr10}

\newcommand{\thmref}[1]{Theorem~\ref{#1}}

\newcommand{\propref}[1]{Proposition~\ref{#1}}
\newcommand{\lemref}[1]{Lemma~\ref{#1}}

\newtheorem{thm}{Theorem}

\newtheorem{lem}[thm]{Lemma}

\newtheorem{prop}[thm]{Proposition}
\newtheorem{rmk}{Remark}[section]

\newtheorem{conj}{Conjecture}

\newcommand{\cF}{{\mathcal{F}}}

\makeatletter
\@namedef{subjclassname@2020}{2020 Mathematics Subject Classification}
\makeatother

\def\({\left(}
\def\){\right)}
\def\[{\left[}
\def\]{\right]}

\def\N{\mathbb{N}}

\def\Q{\mathbb{Q}}

\def\F{\mathbb{F}}

\def\cJ{\mathcal{J}}

\def\cR{\mathcal{R}}

\renewcommand{\mod}{ \text{ mod }}

\newcommand{\Li}{\text{Li}}
\newcommand{\tr}{\text{tr}}

\title{Matsuda monoids and Artin's primitive root conjecture}
\author{Sunil Naik}
\address{
	Department of Mathematics,
	Queen's University, Jeffrey Hall, 
	99 University Avenue, 
	Kingston, ON K7L3N6, 
	Canada}
\email{naik.s@queensu.ca}

\begin{document}
	
\hfuzz 5pt	
	
	
	
	\maketitle

\begin{abstract}
Let $M \subseteq \mathbb{N}_{0}$ be the additive submonoid generated by $2$ and $3$. 
In a recent work,  Christensen,  Gipson and Kulosman proved that 
$M$ is not a Matsuda monoid of type $2$ and type $3$ and 
they have raised the question of whether $M$ is a 
Matsuda monoid of type $\ell$ for any prime $\ell$. 
Assuming the generalized Riemann hypothesis, 
Daileda showed that $M$ is not a Matsuda monoid 
of type $\ell$ for any prime $\ell$. 
In this article, we will establish this result unconditionally using its' connection with Artin's primitive root conjecture and this resolves the question of Christensen, Gipson and Kulosman. 
\end{abstract}	

\section{Introduction and Statements of Results}
The study of irreducible polynomials in various polynomial rings is an important chapter in mathematics. In this context, polynomials with restricted exponents have become the focus of considerable attention in recent years. Motivated by these considerations, R. Matsuda \cite{Ma} introduced polynomial rings
$$
F[X; M] ~=~ \left\{\sum_{i=1}^{n} c_i X^{\alpha_i} ~:~ c_i \in F,~~ \alpha_i \in M,~~ n \in \N \right\},
$$
where $F$ is a field, $M$ is a commutative, torsion free, cancellative additive monoid and began an inquiry into the irreducibility of various polynomials in these rings. 

For example, a non-zero element $\alpha \in M$ is said to be {\bf divisible} if $nz=\alpha$ is solvable in $M$ for some positive integer $n \geq 2$. Otherwise, we say $\alpha$ is {\bf indivisible}. It is easy to see that if $\alpha$ is divisible, then $X^{\alpha} -1 $ is reducible. Indeed, $X^{\alpha} -1 = (x^z - 1)(x^{z(n-1)} + x^{z(n-2)} + \cdots + x^z + 1)$. Therefore, if $X^{\alpha} -1 $ is irreducible, then $\alpha$ must be indivisible. The question arises whether the converse is true. That is, if $\alpha$ is indivisible, then is it true that 
$X^{\alpha} -1 $ is irreducible? This motivates the following.

A cancellative torsion-free commutative monoid $M$ is said to be a Matsuda monoid if for each indivisible $\alpha \in M$, the polynomial $X^{\alpha}-1$ is irreducible in $F[X;M]$ for any field $F$. Let $\ell$ be a prime number. We say that $M$ is a Matsuda monoid of type $0$ (resp. type $\ell$) if for each indivisible $\alpha \in M$, the polynomial $X^{\alpha}-1$ is irreducible in $F[X;M]$ for any field $F$ of characteristic $0$ (resp. characteristic $\ell$). In \cite{Ma}, R. Matsuda showed that any torsion-free abelian group is a Matsuda monoid. In \cite{CGK}, K. Christensen, R. Gipson and H. Kulosman  showed that the  additive submonoid $M$ of $\N_{0}$ generated by $2$ and $3$ is not a Matsuda monoid of type $2$ and type $3$, since the polynomial $x^7-1$ factorizes in $\F_2[X ; M]$ as
$$
X^7 ~-~ 1~=~ (X^4+X^3+X^2+1) (X^3+X^2+1)  
$$
and the polynomial $X^{11}-1$ factorizes in $\F_3[X ; M]$ as
$$
X^{11} ~-~ 1 ~=~ (X^6-X^5-X^4-X^3+X^2+1) (X^5+X^4-X^3+X^2-1)
$$
and $7, 11$ are clearly indivisible elements in $M$.
Further, Christensen, Gipson and Kulosman raised the following question : Whether the monoid $M$ is a Matsuda monoid of any positive type?

Motivated by the above results, one is interested in the following set:
$$
E(\ell) ~=~ \left\{ p \text{ prime } ~:~ X^p -1 \text{ is reducible in } \F_{\ell}[X;M] \right\}.
$$
The above mentioned result of Christensen, Gipson and Kulosman shows that $7 \in E(2)$ and $11 \in E(3)$. We are interested in the following questions: Whether $E(2)$ and $E(3)$ contain infinitely many primes? Is it true that $E(\ell)$ contains infinitely many primes for any prime $\ell$?

Assuming the generalized Riemann hypothesis (GRH), R. C. Daileda \cite{Da} recently showed that the set $E(\ell)$ has a positive lower density for all primes $\ell$. In this paper, we will show {\bf unconditionally} that the set $E(\ell)$ has a positive lower density for all primes $\ell$ and this resolves the question of Christensen, Gipson and Kulosman. More precisely, we prove the following theorem.
\begin{thm}\label{thm2}
Let $\ell$ be a prime number. Then the set
$$
E(\ell) ~=~ \left\{ p \text{ prime } ~:~ X^p -1 \text{ is reducible in } \F_{\ell}[X;M] \right\}
$$
has lower density at least 
$$
1 ~-~ \prod_{p \geq \ell}\(1 - \frac{1}{p(p-1)}\).
$$
Further, the set $E(3)$ has lower density at least
$$
1 ~-~ \prod_{p}\(1 - \frac{1}{p(p-1)}\).
$$
\end{thm}
\begin{rmk}
If one assumes the generalized Riemann hypothesis, then  for $\ell \in \{2, 3\}$, the result of Daileda shows that the density of $E(\ell)$ is $
1 - \prod_{p}\(1 - \frac{1}{p(p-1)}\)$. Thus for $\ell \in \{2, 3\}$, \thmref{thm2} gives the exact lower bound for the density of $E(\ell)$ {\bf unconditionally}.
\end{rmk}

\begin{rmk}
We want to mention to the reader that there are some misprints in section 6 in the article \cite{Da} by Daileda. Note that we are using different techniques to compute lower bounds for the densities of $E(\ell)$ and our results are {\bf unconditional}.
\end{rmk}	

It is interesting to note that the above problem on Matsuda monoids has an intimate connection with Artin's primitive root conjecture. In 1927, E. Artin conjectured that for any integer $a$ other than $\pm 1$ or a perfect square, there exist infinitely many primes $p$ such that $a$ is a primitive root $(\mod p)$ (see subsection \ref{secArtin} for more details). In 1967, assuming the generalized Riemann hypothesis,  C. Hooley \cite{Ho} proved  Artin's conjecture as well as an asymptotic formula for the number of such primes up to $x$. More precisely, under GRH, he showed that

$$
N_a(x) ~=~ \#\{ p \leq x ~:~ a \text{ is a primitive root } (\mod p)\}
~\sim~ A(a) \frac{x}{\log x},
$$ 
as $x \to \infty$. Here
$$
A(a) ~=~ \sum_{d=1}^{\infty} \frac{\mu(d)}{n_d} ~=~ \delta(a) \prod_{p} \(1- \frac{1}{n_p}\),
$$
where $n_d = [ L_d :\Q]$, $L_d = \prod_{p \mid d} \Q(e^{2 \pi i/p}, a^{1/p})$  and
$\delta(a)$ is defined as follows: write $a=bc^2$ with $b$ square-free, then
$$
\delta(a) ~=~ 
\begin{cases}
1 & \text{if } b \not\equiv 1 (\mod 4), \\
1 - \mu(b) \prod_{p \mid b} \frac{1}{n_p -1} & \text{if } b \equiv 1 (\mod 4).
\end{cases}
$$
We will show that the assumption of GRH in the {\bf upper bound} of Hooley's result {\bf can be removed}. More precisely, we prove the following theorem.
\begin{thm}\label{thmArtin}
Let $a$ be a non-zero integer not equal to $\pm 1$ or a perfect square. Then the set of primes $p$ such that $a$  is a primitive root $(\mod p)$
has upper density at most $A(a)$.
\end{thm}

In 2004, S. W. Golomb related a generalization of Artin's primitive root conjecture to M. Ram Murty \cite{Go} (see subsection \ref{secArtin}). 
Assuming GRH, C. Franc and M. Ram Murty \cite{FM} proved that if $a$ is a positive integer which is not an $\ell$-th power for any prime $\ell$, then the density of primes $p \equiv 1 ~(\mod r)$ such that the order of $a~ (\mod p)$ is $(p-1)/r$ is  equal to
$$
A(a,r) ~=~ \sum_{k \geq 1} \frac{\mu(k) m(k)}{r k \varphi(rk)},
$$
where
$$
m(k) ~=~
\begin{cases}
2 & \text{if $rk$ is even and } \sqrt{-a} \in \Q\(\zeta_{2rk}\) \\
1 & \text{otherwise.}
\end{cases}
$$
In this context, unconditionally, we prove the following theorem.
\begin{thm}\label{thmG}
Let $a$ be a positive integer which is not an $\ell$-th power for any prime $\ell$ and let $r$ be a positive integer. Then the set of primes $p$ such that $p \equiv 1 ~(\mod r)$ and the order of $a~ (\mod p)$ is $(p-1)/r$ 
has upper density at most $A(a,r)$.
\end{thm}

The following theorem also plays an important role in the resolution of the question of Christensen, Gipson and Kulosman.
\begin{thm}\label{thmindex}
	Let $\ell$ be prime number. Then the set of primes $p$ such that $[\F_p^\times : \langle \ell \rangle] ~\geq~ \ell$
	has lower density at least
	$$
	1- \prod_{p \geq \ell} \(1- \frac{1}{p(p-1)}\).
	$$
\end{thm}
\begin{rmk}
We remark that in \cite{RM}, M. Ram Murty showed that for a given prime $\ell$, the set of primes $p$ such that $\ell \mid [\F_p^\times : \Gamma_p]$ has density $1/n_\ell(\Gamma)$, where $\Gamma$ is the subgroup of $\Q^\times$ of rank $r$ generated by $a_1, \cdots, a_r$, $\Gamma_p = \Gamma (\mod p)$ and $n_\ell(\Gamma)$ is the degree of $\Q(e^{2\pi i/\ell}, a_1^{1/\ell}, \cdots, a_r^{1/\ell})$ over $\Q$.
\end{rmk}

\medspace

\section{Preliminaries}

\subsection{Prerequisites from primitive roots and Artin's conjecture}\label{secArtin}
Gauss in his {\em Disquisitiones Arithmeticae} noticed that for any prime $p$ not equal to $2$ and $5$, the period of the decimal expansion of $\frac{1}{p}$ is equal to the  smallest natural number $k$ such that
$$
10^k ~\equiv~ 1 ~(\mod p)
$$
and is called the order of $10 ~(\mod p)$. In terms of group-theoretic language, it is equal to the cardinality of the subgroup $\langle 10 \rangle$ of $\F_p^\times$ generated by $10$. Thus by Lagrange's theorem, the largest period of $1/p$ occurs when the order of $10 ~(\mod p)$ is equal to $p-1$. In this case, we say that $10$ is a primitive root $(\mod p)$. More generally, an integer $a$ coprime to $p$ is called a primitive root $(\mod p)$, if $\F_p^\times = \langle a \rangle $. 
A precise conjecture was formulated by E. Artin \cite{Ar} in 1927.
\begin{conj}
For any non-zero integer $a$ not equal to $\pm 1$ or a perfect square, there exist infinitely many primes $p$ for which $a$ is a primitive root $(\mod p)$. Further,
$$
N_a(x) ~=~ \#\{ p \leq x ~:~ a \text{ is a primitive root }(\mod p) \} ~\sim~ A(a) \frac{x}{\log x},
$$
as  $x \to \infty$,
where $A(a)$ is a constant depending on $a$.	
\end{conj}
 This is famously known as Artin's primitive root conjecture. The intuition which led Artin to arrive at this conjecture can be described as follows (please refer to \cite{RM1983,RM1988} for more details): A non-zero integer $a$ is a primitive root $(\mod p)$ if and only if
 $$
 a^{\frac{p-1}{q}} ~\not\equiv~ 1 ~(\mod p)
 $$
 for every prime $q$ dividing $p-1$. Let $\zeta_q$ be a primitive $q$-th root of unity and $L_q =  \Q(\zeta_q, a^{1/q})$, then by Dedekind's theorem, $p$ splits completely in $L_q$ if and only if 
 $$
 a^{\frac{p-1}{q}} ~\equiv~ 1 ~(\mod p).
 $$
 Hence $a$ is a primitive root $(\mod p)$ if and only if $p$ does not split completely in any $L_q$. Note that $p$ splits completely in $L_q$ only when the Artin symbol $\(\frac{L_q/ \Q}{p}\)$ is trivial (see \cite[Ch. 1, p. 18]{La}) and
 by Chebotarev density theorem, the density of the set of primes $p$ which split completely in $L_q$ is $\frac{1}{n_q}$, where $n_q = [L_q : \Q]$. Hence, one heuristically expects that the density of the set of primes $p$ such that $a$ is a  primitive root $(\mod p)$ is
 $$
 \prod_{q} \(1- \frac{1}{n_q}\).
 $$
 However, the above expression did not match with calculations done by D. H. Lehmer and E. Lehmer \cite{LL} for certain values of $a$. Later, a correction factor was proposed by H.  Heilbronn \cite{MW} and was confirmed affirmatively by C. Hooley \cite{Ho} in 1967, subject to the generalized Riemann hypothesis.
 
The first unconditional result in this context was proved in 1984 by R. Gupta and M. Ram Murty \cite{GM} (also see \cite{MS}). They showed that there exists a set of $13$ numbers such that for at least one of these $13$ numbers, Artin's primitive conjecture is true. Later, R. Gupta, M. Ram Murty and V. Kumar Murty \cite{GMM} reduced the size of this set to $7$ and it was subsequently refined to $3$ by D. R. Heath-Brown \cite{HB}. Thus we have the following theorem.
\begin{thm}
One of $2, 3, 5$ is a primitive root $(\mod p)$ for infinitely many primes $p$.
\end{thm} 

In 2004, S. W. Golomb related a generalization of Artin's conjecture to M. Ram Murty \cite{Go}. It states that for every square-free integer $a> 1$ and for every positive integer $r$, there are infinitely many primes $p \equiv 1 ~(\mod r)$ such that the order of $a~ (\mod p)$ is equal to $(p-1)/r$. Moreover, the density of such primes $p$ is equal to a constant  (expressible in terms of $a$ and $r$) times Artin's constant. Assuming GRH, C. Franc and M. Ram Murty \cite{FM} proved the following theorem.
\begin{thm}
Suppose that the generalized Riemann hypothesis is true. Let $a$ be a positive integer that is not an $\ell$-th power for any prime $\ell$ and let $r$ be a positive integer. If $N_{a,r}(x)$ denotes the number of primes $p \leq x$ such that $p \equiv 1 ~(\mod r)$ and the order of $a~ (\mod p)$ is equal to $(p-1)/r$, then 
$$
N_{a,r}(x) ~=~ A(a,r) \frac{x}{\log x} + O\(\frac{x \log\log x}{\log^2 x}\),
$$
where 
$$
A(a,r) ~=~ \sum_{k \geq 1} \frac{\mu(k) m(k)}{r k \varphi(rk)}
$$
and
$$
m(k) ~=~
\begin{cases}
2 & \text{if $rk$ is even and } \sqrt{-a} \in \Q\(\zeta_{2rk}\) \\
1 & \text{otherwise.}
\end{cases}
$$
\end{thm}
In \cite{FM}, C. Franc and M. Ram Murty also showed that $A(a, r) > 0$ if $r$ is odd and $a=bc^2$ with $b$ even and square-free, thereby proving Golomb's conjecture in this special case subject to GRH. In \cite{Mo}, P. Moree  showed that Golomb's conjecture is false for square-free $a \equiv 1 ~(\mod 4)$ and odd $r$ divisible by $a$. 

Now we will state a few results related to effective Chebotarev density theorem which we require in proofs of main results. Let $L$ be a finite Galois extension of $\Q$ with Galois group $G$. Also let $n_L, D_L$ denote the  degree and the absolute value of the discriminant of $L$ respectively. The Dedekind zeta function $\zeta_L(s)$ has at most one zero in the region defined by $s= \sigma+it$ with
$$
1 - \frac{1}{4 \log d_L} ~\leq~ \sigma ~\leq~ 1,
\phantom{mm}
|t| ~\leq~ \frac{1}{4 \log d_L}
$$
If such a zero  $\beta$ exists, it must be real and simple and is usually called a Siegel zero. We have the following effective version of Chebotarev density theorem due to J. C. Lagarias and A. M. Odlyzko \cite[Theorem 1.3]{LO}.
\begin{thm}\label{thmLO}
Let $L$ be a finite Galois extension of $\Q$ with Galois group $G$ and  $C$ be a conjugacy class in $G$. Also let
$$
\pi_C(x) ~=~ \#\left\{ p \leq x ~:~ p \text{ is unramified in } L,~~ \(\frac{L/\Q}{p} \) = C\right\}.
$$
There exist absolute effectively computable constants $c_1$ and $c_2$ such that if $x \geq \exp\( 10n_L (\log D_L)^2 \)$, then 
$$
\left| \pi_C(x) - \frac{|C|}{|G|} \Li(x) \right|
~\leq~
\frac{|C|}{|G|} \Li(x^{\beta}) + c_1 x \exp\(-c_2 \sqrt{\frac{\log x}{n_L}}\),
$$
where the term involving $\beta$ is present only when $\beta$ exists.
\end{thm}
The best known bound for $\beta$ is due to H. M. Stark  \cite[p. 148]{St}.
\begin{thm}\label{thmStb}
There exists an absolute constant $c_3$ such that 
$$
\beta ~<~ \max\(1 - \frac{1}{4 \log D_L},~ 1- \frac{c_3}{ D_L^{1/n_L}} \).
$$
\end{thm}
We have the following bound for $D_L$ in terms of $n_L$ \cite[Prop. 6, p. 130]{Se} due to J.-P. Serre.
\begin{prop}\label{propDL}
We have
$$
\frac{n_L}{2} \sum_{p \mid D_L} \log p ~\leq~
\log D_L ~\leq~
(n_L -1) \sum_{p \mid D_L} \log p ~+~ n_L \log n_L.
$$
\end{prop}

\medspace

\subsection{Interplay between Matsuda monoids and primitive roots}\label{subSecMArt}
Let $\ell$ be a prime number and $\F_{\ell}$ be the finite field of $\ell$ elements. Let $n \in \N $ and for a polynomial $f(X)= a_n X^n + a_{n-1} X^{n-1} + \cdots + a_1 X + a_0  \in \F_{\ell}[X]$ with $a_n \neq 0$, the trace of $f$ is given by
$$
\text{tr}(f) ~=~ -\frac{a_{n-1}}{a_n}.
$$
We say that $f$ is a traceless polynomial if $\text{tr}(f) = 0$.
For $f \in \F_{\ell}[X] \backslash\{0\}$, let
$$
\cR(f)(X) ~=~ X^{\text{deg}(f)}f\(\frac{1}{X}\).
$$
Let $M$ be the  additive submonoid of $\N_{0}$ generated by $2$ and $3$. Then we have $M = \N_{0} \backslash\{1\}$. Further, we have
$$
\F_{\ell}[X;M] ~=~ \left\{ a_0 + a_2 X^2 + \cdots + a_n X^n ~:~
a_i \in \F_{\ell},~~ 0 \leq i \leq n,~~ n \in \N_0 \backslash\{1\} \right\} \subset \F_{\ell}[X].
$$
Let $f(X), g(X) \in \F_{\ell}[X;M]$ be non-constant polynomials which are non-vanishing at $0$ i.e., $f(0) \neq 0$, $g(0) \neq 0$. Then we have
\begin{itemize}
	\item[i)]$ \text{tr}(\cR(f))~=~ 0$,
	\item[ii)] $\cR (fg) ~=~ \cR(f) \cR(g)$,
	\item[iii)] $\cR\(\cR(f)\) ~=~ f$.
\end{itemize}
Thus a non-constant polynomial $f(X)\in \F_{\ell}[X;M]$ with $f(0) \neq 0$ is reducible in $\F_{\ell}[X;M]$ if and only if $\cR(f)$ factors in $\F_{\ell}[X]$ as a product of traceless polynomials. In particular, we have the following lemma.
\begin{lem}\label{lemredFlM}
Let $n \geq 2$ be a natural number, then $X^n -1$ is reducible in $ \F_{\ell}[X;M]$ if and only if it factors as a product of traceless polynomials in $\F_{\ell}[X]$.
\end{lem}
Let $p$ be a prime number and $\Phi_p(X)$ denote the $p$-th cyclotomic polynomial. Then we have
$$
\Phi_p(X) ~=~ \frac{X^p - 1}{X-1} ~=~ 1+ X+ \cdots + X^{p-1}
$$
and $\text{tr}(\Phi_p) = -1$. As a consequence of \lemref{lemredFlM}, we have the following lemma.
\begin{lem}\label{lemredXnPhip0}
	Let $ p$ and $\ell$ be distinct prime numbers. Then $X^p - 1$ is reducible in $ \F_{\ell}[X;M]$ if and only if $\Phi_p(X)$ has a traceless factor in $\F_{\ell}[X]$.
\end{lem}
We have the following  lemma on the factorization of $\Phi_p(X)$ in $\F_{\ell}[X]$ (see \cite[Theorem 2.47, p. 65]{LN}).
\begin{lem}\label{lemfactPhi}
Let $\ell$ and $p$ be distinct prime numbers and let  $r$ be the order of $\ell ~(\mod p)$. Then $\Phi_p(X)$ factors into $(p-1)/r$ distinct monic irreducible polynomials in $\F_{\ell}[X]$ of the same degree $r$.  
\end{lem}
Let
$$
E(\ell) ~=~ \left\{ p  ~:~ X^p -1 \text{ is reducible in } \F_{\ell}[X;M] \right\}.
$$ 
From \lemref{lemredXnPhip0}, we have $p \in E(\ell)$ if and only if $\Phi_p(X)$ has a traceless factor in $\F_{\ell}[X]$.
From \lemref{lemfactPhi}, $\Phi_p(X)$ is irreducible in $\F_{\ell}[X]$ if and only if $\ell$ is a primitive root $(\mod p)$. Further, if $\Phi_p(X)$ is reducible in $\F_{2}[X]$, then one can easily show that $\Phi_p(X)$ has a traceless factor in  $\F_{2}[X]$ and the same result is true when $\ell = 3$. Also note that $\ell \not\in E(\ell)$. Hence we have the following lemma (see \cite[Lemma 3]{Da}).
\begin{lem}\label{lemE23}
	Let $\ell \in \{2, 3\}$, then we have
	$$
	E(\ell) ~=~ \left\{ p ~:~ \ell \text{  is not a primitive root } (\mod p) \right\}.
	$$
\end{lem}
However, Daileda showed that the above lemma may not hold if $\ell \geq 5$. For instance, $5$ is not a primitive root $(\mod 11)$ but $11 \not\in E(5)$.

\medspace

\section{Artin's primitive root conjecture: Upper bound}
In this section, we will give a proof of \thmref{thmArtin} following Hooley's method.
\subsection{Proof of \thmref{thmArtin}} Let $a$ be a non-zero integer not equal to $\pm 1$ or a perfect square and
$$
N_a(x) ~=~ \#\{ p  \leq x ~:~ a \text{ is a primitive root } (\mod p)\}.
$$
A necessary and sufficient condition for $a$ to be a  primitive root $(\mod p)$ is 
$$
a^{\frac{p-1}{q}} ~\not\equiv~ 1 ~(\mod p)
$$
for every prime $q$ dividing $p-1$. Let $L_q = \Q(\zeta_q, a^{1/q})$, where $\zeta_q$ denotes a primitive $q$-th root of unity. Then $L_q$ is a Galois extension over $\Q$. According to a principle of Dedekind, 
$$
p ~\text{ splits completely in } L_q 
~\iff~ 
a^{\frac{p-1}{q}} ~\equiv~ 1 ~(\mod p).
$$
For any square-free positive integer $k$, let
$$
L_k ~=~ \prod_{q \mid k} L_q
$$
be the compositum of the fields $L_q$ for primes $q \mid k$. Note that
$$
N_a(x)  ~\leq~ \#\{ p \leq x ~:~ p \text{ does not split completely in } L_d \text{ for any } d \mid k, d > 1 \}.
$$
By the inclusion-exclusion principle, we deduce that for any $k$,
\begin{equation}\label{eqUNa}
N_a(x) ~\leq~ \sum_{d \mid k} \mu(d) \pi_d(x),
\end{equation}
where
$$
\pi_d(x) ~=~ \#\{ p \leq x ~:~ p \text{ splits completely in } L_d\}.
$$
As mentioned in the subsection \ref{secArtin}, $p$ splits completely in $L_d$ if and only if $p$ is unramified in $L_d$ and the Artin symbol $\(\frac{L_d/\Q}{p}\)$ is trivial. By a result of Lagarias and Odlyzko (see \thmref{thmLO}), there exist absolute constants $c_1$ and $c_2$ such that if $x \geq \exp\(10 n_d (\log D_d)^2\)$, then 
\begin{equation}\label{eqLO}
\left| \pi_d(x) - \frac{\Li(x)}{n_d} \right| 
~\leq~ 
\frac{\Li(x^{\beta_d})}{n_d} + 
c_1x \exp\(-c_2 \sqrt{\frac{\log x}{n_d}} \),
\end{equation}
where $\beta_d$ denotes the Siegel zero of the Dedekind zeta function $\zeta_{L_d}(s)$ (if it exists).
Here $n_d= [L_d : \Q]$ and $D_d$ denotes the absolute value of the discriminant of $L_d$.
Note that for any $d \mid k$,
\begin{equation}\label{eqund}
n_d ~\leq~ \prod_{q \mid d} n_q ~\leq~ \prod_{q \mid d} q(q-1) ~=~ d \varphi(d). 
\end{equation}
From \cite[Eq. 12]{Ho}, we get
\begin{equation}\label{eqlnd}
n_d ~\gg~ d \varphi(d),
\end{equation}
where the implied constant may depend on $a$.
By \propref{propDL} and \eqref{eqund}, we get
\begin{equation}
\log D_d ~\leq~ d \varphi(d)  \sum_{p \mid ad} \log p ~+~ d \varphi(d) \log (d \varphi(d)) ~\ll~ d \varphi(d) \log d,
\end{equation}
where the implied constant depends only on $a$. Thus we have
\begin{equation}
n_d (\log D_d)^2 ~\ll~ (d\varphi(d))^3 \log^2 d.
\end{equation}
By Stark's bound (see \thmref{thmStb}), \eqref{eqund}  and \eqref{eqlnd}, we get
\begin{equation}\label{equbeta}
\beta_d ~<~ \max\(1 - \frac{1}{4 \log D_d},~ 1- \frac{c_3}{ D_d^{1/n_d}}\)
~<~  \max\(1 - \frac{\tilde{c}_3}{d \varphi(d) \log d},~ 1- \frac{\tilde{c}_3}{d^{c_4}}\).
\end{equation}
for some positive constants $\tilde{c}_3$ and $c_4 > 1$ depending on $a$. For any integer $n > 1$, set $\log_n x = \log\(\log_{n-1} x\)$ and $\log_1 x = \log x$.
We choose
$$
z ~=~ \frac{\log\log x}{8c_4} \phantom{mm}\text{and}\phantom{mm} k ~=~ \prod_{ q \leq z} q.
$$
Then for any $d \mid k$, we have 
$$
\exp\(10 n_d (\log D_d)^2\)~\leq~ x
$$
for all sufficiently large $x$.  Hence from \eqref{eqLO}, we get
\begin{eqnarray}\label{eqAsypid}
\pi_d(x) &~=~& \frac{\Li(x)}{n_d} ~+~
O\(\frac{x}{n_d} \exp\(- (\log x)^{4/7}\)\) ~+~ 
O\( x \exp\( -(\log x)^{1/3}\) \) \nonumber \\ 
&~=~& \frac{\Li(x)}{n_d} ~+~ O\( x \exp\(-(\log x)^{1/3}\) \).
\end{eqnarray}
From \eqref{eqUNa} and \eqref{eqAsypid}, we get
\begin{eqnarray*}
N_a(x) 
&~\leq~& \Li(x) \sum_{d \mid k} \frac{\mu(d)}{n_d} 
~+~ O\( \sum_{d \mid k} x \exp\(-(\log x)^{1/3}\)\)  \nonumber \\
&~=~& \Li(x) \sum_{d \mid k} \frac{\mu(d)}{n_d} 
~+~ O\(x \exp\(-(\log x)^{1/4}\)\),
\end{eqnarray*}
since the number of divisors of $k$ is $2^{\pi(z)}$ which is negligible in the summation.
By partial summation, we get
\begin{eqnarray*}
\sum_{d \mid k} \frac{\mu(d)}{n_d} 
&~=~& \sum_{d \leq z} \frac{\mu(d)}{n_d} ~+~ \sum_{d \mid k \atop d > z} \frac{\mu(d)}{n_d} 
~=~\sum_{d=1}^{\infty} \frac{\mu(d)}{n_d} ~+~ 
O\(\sum_{d > z} \frac{\log\log d}{d^2}\)  \\
&~=~& \sum_{d=1}^{\infty} \frac{\mu(d)}{n_d} ~+~
O\( \frac{\log_4 x}{\log\log x}\)
~=~ A(a)  ~+~
O\( \frac{\log_4 x}{\log\log x}\).
\end{eqnarray*} 
Hence we deduce that 
$$
N_a(x) ~\leq~ A(a)~ \Li(x) 
~+~ O\( \frac{x \log_4 x}{\log x \log_2 x}\).
$$
This completes the proof of \thmref{thmArtin}. \qed

\medspace

\subsection{Proof of \thmref{thmG}}
Let $a$ be a positive integer which is not an $\ell$-th power for any prime $\ell$ and let $r$ be a positive integer. Also let
$$
N_{a,r}(x) ~=~ \#\{ p \leq x ~:~ p \equiv 1 ~(\mod r) \text{ and order of } a ~(\mod p)~ \text{ is } (p-1)/r\}.
$$
Let $L_r = \Q(\zeta_r, a^{1/r})$ and $L_{rq}  = \Q(\zeta_{rq}, a^{1/rq})$ for any prime $q$. For any square-free positive integer $k$, let
$$
L_{rk} ~=~ \prod_{q \mid k} L_{rq}
$$
be the compositum of the fields $L_{rq}$ for primes $q \mid k$. By a principle of Dedekind, we have $p ~\equiv~ 1 ~(\mod r)$  and order of  $a ~(\mod p)$ is equal to $(p-1)/r$ if and only if $p$ splits completely in $L_r$ but not in $L_{rq}$ for any prime $q$. Note that
\begin{equation*}
N_{a,r}(x) ~\leq~ \#\{ p \leq x ~:~ p \text{ splits completely in $L_r$ but not in $L_{rd}$  for } d \mid k,~ d>1 \}
\end{equation*}
By the inclusion-exclusion principle, we get, as before,
\begin{equation}\label{eqInc-ExcLrd}
N_{a,r}(x) ~\leq~ \sum_{d \mid k} \mu(d) \pi_{rd}(x),
\end{equation}
where
$$
\pi_{rd}(x) ~=~\#\{p \leq x ~:~ p \text{ splits completely in } L_{rd}\}.
$$
Let $n_{rd} ~=~ [L_{rd} : \Q]$ and $D_{rd}$ denote the absolute value of the discriminant of $L_{rd}$. By \thmref{thmLO}, if $x \geq \exp\(10 n_{rd} (\log D_{rd})^2\)$, then
\begin{equation}\label{eqChebLrd}
\left|
\pi_{rd}(x) - \frac{\Li(x)}{n_{rd}}
\right|
~\leq~
 \frac{\Li(x^{\beta_{rd}})}{n_{rd}}
 ~+~ c_1 \exp\(-c_2 \sqrt{\frac{\log x}{n_{rd}}}\).
\end{equation}
From \cite[p. 1283]{FM}, we have
\begin{equation}\label{eqnrd}
n_{rd} ~=~ \frac{rd \varphi(rd)}{m(d)}, \phantom{mm} m(d) ~=~
\begin{cases}
	2 & \text{if $rd$ is even and } \sqrt{-a} \in \Q\(\zeta_{2rd}\) \\
	1 & \text{otherwise.}
\end{cases}
\end{equation}
Thus $n_{rd} \leq rd \varphi(rd)$ and by \propref{propDL},
\begin{equation}\label{eqDrd}
\log (D_{rd}) ~\ll~ rd \varphi(rd) \log (rd),	
\end{equation}
where the implied constant depends only on $a$.
By using Stark's bound (see \thmref{thmStb}), we get
\begin{equation}\label{eqbetard}
\beta_{rd} ~\leq~ \max\( 1- \frac{c_5}{(rd)^2 \log(rd)},~ 1 -\frac{c_5}{(rd)^{c_6}} \),
\end{equation}
where $c_5$ and $c_6>1$ are positive constants depending on $a$. We choose
$$
z~=~ \frac{\log\log x}{8c_6} ~-~ 2\log r 
\phantom{mm}\text{and}\phantom{mm} 
k~=~ \prod_{q \leq z} q
$$
and we suppose that $x$ is sufficiently  large. By \eqref{eqChebLrd}, \eqref{eqnrd}, \eqref{eqDrd} and \eqref{eqbetard}, we get
\begin{equation}\label{eqpirdasy}
\pi_{rd}(x) ~=~ \frac{\Li(x)}{n_{rd}} ~+~O\(x \exp(-(\log x)^{1/3})\)
\end{equation}
for all sufficiently large $x$. By \eqref{eqInc-ExcLrd} and \eqref{eqpirdasy} and using partial summation, we get
\begin{eqnarray*}
N_{a,r}(x) 
&~\leq~& 
\Li(x) \sum_{d \mid k } \frac{\mu(d)}{n_{rd}} ~+~ O\(x \exp(-(\log x)^{1/4})\) \\
&~=~& 
\Li(x) \sum_{d=1}^{\infty} \frac{\mu(d) m(d)}{rd \varphi(rd)}  ~+~O\( \frac{x \log_4 x}{\log x \log_2 x}\) \\
&~=~&
A(a,r)~ \Li(x) ~+~ O\( \frac{x \log_4 x}{\log x \log_2 x}\)
\end{eqnarray*}
for all sufficiently large $x$.
This completes the proof of \thmref{thmG}. \qed

\medspace

\section{The density of $E(\ell)$}
Let $\ell$ be a prime number and $r$ be a positive integer.  Set
$$
I_r(\ell) ~=~ \left\{ p ~:~ [\F_{p}^\times : \langle \ell \rangle ] = r\right\}
\phantom{mm}\text{and}\phantom{mm}
E_r(\ell) ~=~ \left\{ p ~:~ p \in E(\ell) \text{ and } [\F_{p}^\times : \langle \ell \rangle ] = r \right\}.
$$
It is clear that $E_r(\ell) \subseteq I_r(\ell)$ for any positive integer $r$ and we will show that the equality holds if $r$ is sufficiently large (see the remarks in \cite[Sec. 7]{Da}). More precisely, we prove the following lemma.
\begin{lem}\label{lemErIr}
Let $\ell$ be a prime number and $r$ be a positive integer. 
If $r \geq \ell$, then we have $E_r(\ell) = I_r(\ell)$.
\end{lem}

\begin{proof}
Suppose that $r \geq \ell$. It is sufficient to prove that $I_r(\ell) \subseteq E_r(\ell)$. Let $p \in I_r(\ell)$ be a prime number. By \lemref{lemredXnPhip0}, it is sufficient to show that $\Phi_p(X)$ has a traceless factor in $\F_{\ell}[X]$.
Since $[\F_{p}^\times : \langle \ell \rangle ] = r$, by \lemref{lemfactPhi}, the polynomial $\Phi_p(X)$ factors into $r$ distinct irreducible factors in $\F_{\ell}[X]$, say
$$
\Phi_p ~=~ f_1 f_2 \cdots f_r.
$$
Let $a_i = \tr(f_i) \in \F_{\ell}$ for $1 \leq i \leq r$. Let $D(\F_{\ell})$ denote the Davenport constant of $\F_{\ell}$ which is defined to be the smallest positive integer $n$ such that any sequence in $\F_{\ell}$ of length at least $n$ contains a sub-sequence whose sum is $0$.
We know that $D(\F_{\ell}) = \ell$ (see \cite[Theorem 1]{Ro}). Since $r \geq \ell$, the sequence $(a_i)_{i=1}^{r}$ contains a zero-sum sub-sequence i.e., there exists a subset $ \cJ \subseteq \{1, 2, \cdots, r\}$ such that
$$
\sum_{j \in \cJ} a_{j} ~=~ 0.
$$
This implies that 
$$
\tr\(\prod_{j \in \cJ} f_j\) ~=~ \sum_{j \in \cJ} a_{j} ~=~ 0.
$$
This completes the proof of \lemref{lemErIr}.
\end{proof}

\medskip
The following lemma gives a lower bound for the density of $E(\ell)$ whenever $\ell \in \{2, 3\}$.

\begin{lem}\label{lemell23}
Let $\ell \in \{2, 3\}$, then the set 
$$
E(\ell)  ~=~ \left\{ p  ~:~ X^p -1 \text{ is reducible in } \F_{\ell}[X;M] \right\}
$$ 
has lower density at least
$$
1~-~ \prod_{p}\(1- \frac{1}{p(p-1)}\).
$$
\end{lem}
\begin{proof}
Suppose that $\ell \in \{2, 3\}$. By \lemref{lemE23}, we know that $p \in E(\ell)$ if and only if $\ell$ is not a primitive root $(\mod p)$. By \thmref{thmArtin}, the set of primes $p$ such that $\ell$ is a primitive root $(\mod p)$ has upper density at most $A(\ell)=A$, where $A$ is Artin's constant given by
$$
A ~=~ \prod_{p}\(1- \frac{1}{p(p-1)}\).
$$
Hence we deduce that the set $E(\ell)$ has lower density at least $1- A$ whenever $\ell \in \{2,3\}$.
\end{proof}

\medspace

Let $q$ be a prime number. We remark that \cite[Lemma 4]{RM} gives an asymptotic expression for the number of primes $p \leq x$ such that $q \mid [\F_p^\times : \Gamma_p]$, where $\Gamma$ is a finitely generated subgroup of $\Q^\times$ and  $\Gamma_p = \Gamma ~(\mod p)$. Using similar techniques and the inclusion-exclusion principle, we will give a proof \thmref{thmindex}.

\subsection{Proof of \thmref{thmindex}}
Let $\ell$ be a prime number. If $\ell = 2$, then
$$
\left\{ p  ~:~ [\F_p^\times : \langle 2 \rangle] \geq 2 \right\}
~=~
\{ p ~:~ 2 \text{ is not a primitive root }(\mod p)\}.
$$
Then, by \thmref{thmArtin}, the set of primes $p$ such that $ [\F_p^\times : \langle 2 \rangle] \geq 2$ has lower density at least 
$$
1~-~A(2) ~=~ 1 ~-~ \prod_{p} \(1 - \frac{1}{p(p-1)}\).
$$
Henceforth, we assume that $\ell$ is an odd prime.
Note that for any prime $q$,
\begin{eqnarray*}
q \mid [\F_p^\times : \langle \ell \rangle] 
&~\iff~&
\ell^{\frac{p-1}{q}} ~\equiv~ 1 ~(\mod p) \\
&~\iff~&
p \text{ splits completely in } K_q = \Q(\zeta_q, \ell^{1/q}).
\end{eqnarray*}
For any square-free natural number $k$, let
$$
K_k ~=~ \prod_{q \mid k} K_q
$$
be the composition of the fields $K_q$ for primes $q \mid k$. Let $m_k$ and $D_k$ denote the degree and the absolute value of the discriminant of $K_k$ respectively. We first note that for any prime $q \geq \ell$,
$$
q \mid [\F_p^\times : \langle \ell \rangle]  ~\implies~ [\F_p^\times : \langle \ell \rangle] ~\geq~ \ell
$$
We choose
$$
k ~=~ \prod_{\ell \leq q \leq z} q,
$$
where $z$ is a positive real number which will be chosen later. Note that
$$
\#\left\{ p \leq x ~:~ [\F_p^\times ~:~ \langle \ell \rangle] ~\geq~ \ell \right\}
~\geq~ 
\#\left\{ p \leq x ~:~ q \mid [\F_p^\times ~:~ \langle \ell \rangle] \text{ for some prime } q \mid k\right\}.
$$
Let $S_\ell(x) = \#\{ p \leq x ~:~ [\F_p^\times : \langle \ell \rangle] \geq \ell \}$. By the inclusion-exclusion principle, we deduce that
\begin{equation}\label{eqlowIndell}
S_\ell(x) 
~\geq~ 
- \sum_{d \mid k \atop d>1} \mu(d) \pi_d(x),
\end{equation}
where
$$
\pi_d(x) ~=~\#\{ p \leq x ~:~ p \text{ splits completely in } K_d\}.
$$
By \thmref{thmLO}, if $x \geq \exp\(10m_d (\log D_d)^2\)$, then we have
\begin{equation}\label{eqLOK}
\left| \pi_d(x) - \frac{\Li(x)}{m_d} \right| 
~\leq~ 
\frac{\Li(x^{\beta_d})}{m_d} + 
c_1x \exp\(-c_2 \sqrt{\frac{\log x}{m_d}} \),
\end{equation}
where $\beta_d$ denotes the Siegel zero of the Dedekind zeta function $\zeta_{K_d}(s)$ (if it exists).
By \cite[Eq. 12]{Ho}, we have $m_d = d \varphi(d)$ for any $d \mid k$ and by \propref{propDL}, we get
\begin{equation}\label{eqDiscKd}
\log D_d ~\ll~ d \varphi(d) \log d,
\end{equation}
where the implied constant depends on $\ell$. Hence
$$
m_d (\log D_d)^2 ~\ll~ d^6 (\log d)^2 
\phantom{mm}\text{and}\phantom{mm}
D_d^{1/m_d} ~=~ \exp\(\frac{\log D_d}{m_d}\) ~\leq~ d^{c_5} 
$$
for some constant $c_5 > 1$ (depending on $\ell$).
We now chose
$$
z ~=~ \frac{\log\log x}{8c_5}.
$$
Then for any $d \mid k$, we get 
$$
\exp\(10m_d (\log D_d)^2\) ~\leq~ x  \phantom{mm}\text{and}\phantom{mm}
\Li(x^{\beta_d}) ~\leq~ x \exp\(-(\log x)^{4/7}\)
$$
for all sufficiently large $x$. Thus from \eqref{eqlowIndell} and \eqref{eqDiscKd}, we deduce that
\begin{eqnarray*}
S_\ell(x) 
&~\geq~& 
-\sum_{d \mid k \atop d>1} \mu(d) 
\(\frac{\Li(x)}{d \varphi(d)} ~+~ x \exp(-(\log x)^{1/3})\) \\
&~=~&
-\Li(x) \sum_{d \mid k \atop d>1} \frac{\mu(d)}{d \varphi(d)} 
~+~ O\(x \exp(-(\log x)^{1/4})\) \\
&~=~&
\Li(x) \(1- \sum_{d \mid k } \frac{\mu(d)}{d \varphi(d)} \)
~+~ O\(x \exp(-(\log x)^{1/4})\).
\end{eqnarray*}
By partial summation, we deduce that
$$
\sum_{d \mid k } \frac{\mu(d)}{d \varphi(d)} ~=~ \sum_{d=1 \atop p \mid d \Rightarrow p \geq \ell}^{\infty} \frac{\mu(d)}{d \varphi(d)} ~+~ O\(\frac{\log_4 x}{\log\log x}\).
$$
Hence we conclude that the set of primes $p$ such that $[\F_p^\times : \langle \ell \rangle] \geq \ell$ has lower density at least
$$
1 ~-~ \sum_{d=1 \atop p \mid d \Rightarrow p \geq \ell}^{\infty} \frac{\mu(d)}{d \varphi(d)} ~=~ 1 ~-~ \prod_{p \geq \ell} \(1- \frac{1}{p(p-1)}\).
$$
This completes the proof of \thmref{thmindex}. \qed

\medspace

Now we are ready to give a proof of \thmref{thm2}.
\subsection{Proof of \thmref{thm2}}
Let $\ell$ be a prime number and 
$$
E(\ell) ~=~ \left\{ p  ~:~ X^p -1 \text{ is reducible in } \F_{\ell}[X;M] \right\}.
$$
By \lemref{lemredXnPhip0}, if $p \in E(\ell)$, then $\Phi_p(X)$ has a traceless factor in $\F_{\ell}[X]$ and in particular, $\Phi_p(X)$ is reducible in $\F_{\ell}[X]$, since $\tr(\Phi_p) = -1$. Thus by \lemref{lemfactPhi}, $[\F_p^{\times} : \langle \ell \rangle] \geq 2$ whenever $p \in E(\ell)$. Thus we have
$$
E(\ell) ~=~ \bigcup_{r=2}^{\infty} E_r(\ell).
$$
By \lemref{lemErIr}, we get
$$
E(\ell) ~\supseteq~ 
\bigcup_{r= \ell}^{\infty} I_r(\ell) 
~=~ \left\{ p ~:~ [\F_p^{\times} ~:~ \langle \ell \rangle] ~\geq~ \ell \right\}.
$$
Now from \thmref{thmindex}, we deduce that the set $E(\ell)$ has lower density at least
$$
 1 ~-~ \prod_{p \geq \ell} \(1- \frac{1}{p(p-1)}\).
$$
This completes the proof of \thmref{thm2}. \qed

\section{Concluding remarks}
In a plethora of Artin-type questions, one will be estimating the density of the set of primes $p$ that split completely in $K$ but not in $L_\alpha$ for every $\alpha \in \cF$, where $K$ and $L_\alpha$ are finite Galois extensions of $\Q$ and $\cF$ is an infinite indexing set. Hence, one can apply the method used in the proof of \thmref{thmArtin} to derive an upper bound for the density of the set of such primes unconditionally.

\section*{Acknowledgments}
The author would like to thank M. Ram Murty for his guidance, support, and suggestion of removing the generalized Riemann hypothesis in Hooley's upper bound and other helpful suggestions on an earlier version of this article. The author would also like to thank Queen's University, Canada, for providing an excellent atmosphere to work.

\medspace


\begin{thebibliography}{100}
	
\bibitem{Ar}
E. Artin, 
Collected papers, {\em Addison-Wesley}, 1965.

\bibitem{CGK}
K. Christensen, R. Gipson and H. Kulosman, 
{\em Irreducibility of certain binomials in semigroup rings for nonnegative rational monoids}, 
Int. Electron. J. Algebra {\bf 24} (2018), 50--61.	

\bibitem{Da}
R. C. Daileda,
{\em An application of a generalization of Artin's primitive root conjecture in the theory of monoid rings}, 
https://doi.org/10.48550/arXiv.2112.09080.

\bibitem{FM}
C. Franc and M. Ram Murty, 
{\em On a generalization of Artin's conjecture}, 
Pure Appl. Math. Q. {\bf 4} (2008), no. 4, Special Issue: In honor of Jean-Pierre Serre. Part 1, 1279--1290.

\bibitem{Go}
S. W. Golomb, 
Letter to M. Ram Murty, June 22, 2004.

\bibitem{GM}
R. Gupta and M. Ram Murty, 
{\em A remark on Artin's conjecture}, 
Invent. Math. {\bf 78} (1984), no. 1, 127--130.

\bibitem{GMM}
R. Gupta, M. Ram Murty and V. Kumar Murty, 
{\em The Euclidean algorithm for $S$-integers}, CMS Conf. Proc. {\bf 7}  (1985), 189--201.

\bibitem{HB}
D. R. Heath-Brown,
{\em Artin's conjecture for primitive roots}, 
Quart. J. Math. Oxford Ser. (2) {\bf 37} (1986), no. 145, 27--38.

\bibitem{MW}
H. Heilbronn, See:
J. C. P. Miller and A. E. Western, 
Tables of indices and primitive roots,
Royal Society Mathematical Tables, Vol. {\em 9}, 
Published for the Royal Society at the Cambridge University Press, London, 1968.

\bibitem{Ho}
C. Hooley, 
{\em On Artin's conjecture}, 
J. Reine Angew. Math. {\bf 225} (1967), 209--220.

\bibitem{LO}
J. C. Lagarias and A. M.  Odlyzko, 
{\em Effective versions of the Chebotarev density theorem}, 
Algebraic number fields: L-functions and Galois properties (Proc. Sympos., Univ. Durham, Durham, 1975), pp. 409--464, Academic Press, London-New York, 1977.

\bibitem{La}
S. Lang, 
Algebraic number theory, Second edition, 
Graduate Texts in Mathematics {\bf 110}, 
{\em Springer-Verlag, New York}, 1994.

\bibitem{LL}
D. H. Lehmer and E. Lehmer,
Heuristics, anyone?, 
Studies in mathematical analysis and related topics, pp. 202--210,
Stanford Studies in Mathematics and Statistics, IV, 
{\em Stanford Univ. Press, Stanford, CA}, 1962.

\bibitem{LN}
R. Lidl and H. Niederreiter, 
Finite fields, Second edition, Encyclopedia of Mathematics and its Applications, {\bf 20}, {\em Cambridge University Press, Cambridge}, 1997. 
	
\bibitem{Ma}
R. Matsuda, 
{\em On algebraic properties of infinite group rings}, 
Bull. Fac. Sci. Ibaraki Univ. Ser. A No. {\bf 7} (1975), 29--37. 

\bibitem{Mo}
P. Moree, 
{\em Near-primitive roots}, 
Funct. Approx. Comment. Math. {\bf 48} (2013), part 1, 133--145.

\bibitem{RM1983}
M. Ram Murty, 
{\em On Artin's conjecture}, 
J. Number Theory {\bf 16} (1983), no. 2, 147--168.

\bibitem{RM1988}
M. Ram Murty, 
{\em Artin's conjecture for primitive roots}, 
Math. Intelligencer {\bf 10} (1988), no. 4, 59--67.

\bibitem{RM}
M. Ram Murty, 
{\em Finitely generated groups $(\mod p)$}, 
Proc. Amer. Math. Soc. {\em 122} (1994), no. 1, 37--45.

\bibitem{MS}
M. Ram Murty and S. Srinivasan, 
{\em Some remarks on Artin's conjecture}, 
Canad. Math. Bull. {\bf 30} (1987), no. 1, 80--85.

\bibitem{Ro}
K. Rogers, 
{\em A combinatorial problem in Abelian groups}, 
Proc. Cambridge Philos. Soc. {\bf 59} (1963), 559--562.

\bibitem{Se}
J.-P. Serre, 
{\em Quelques applications du th\'eor\`eme de densit\'e de Chebotarev}, 
Inst. Hautes \'Etudes Sci. Publ. Math, No. {\bf 54} (1981), 323--401.

\bibitem{St}
H. M. Stark, 
{\em Some effective cases of the Brauer-Siegel theorem}, 
Invent. Math. {\bf 23} (1974), 135--152.

\end{thebibliography}
\end{document}